\documentclass[10pt]{article}

\usepackage{graphicx}
\usepackage{amsmath,amsthm,amsfonts,amssymb,amscd,enumerate,pictex}

\theoremstyle{plain}

\newtheorem{subSCConjecture}[subsection]{Singer's Conjecture for Coxeter groups}

\newtheorem*{Mainthm}{Main Theorem}
\newtheorem{Theorem}[subsection]{Theorem}

\newtheorem{Corollary}[subsection]{Corollary}
\newtheorem{Lemma}[subsection]{Lemma}

\theoremstyle{definition}

\newtheorem{Defn}[subsection]{Definition}

\newcommand{\cs}{\mathcal{S}}

\newcommand{\cH}{\mathcal{H}}

\newcommand{\mfh}{\mathfrak{h}}

\def\l{\operatorname{\ell}}
\newcommand{\ltwo}{\l^2}

\newcommand{\Edge}{\operatorname{Edge}}

\newcommand{\gS}{\Sigma}

\newcommand{\gb}{\beta}
\newcommand{\gs}{\sigma}

\newcommand{\BS}{\mathbb{S}}

\newcommand{\BN}{\mathbb{N}}

\numberwithin{equation}{section}
\begin{document}

\title{$\ell^2$-homology and planar graphs}

\author{Timothy A. Schroeder}

\date{\today}
\maketitle

\begin{abstract}
In his 1930 paper \cite{kuratowski}, Kuratowksi categorized planar graphs, proving that a finite graph $\Gamma$ is planar if and only if it does not contain a subgraph that is homeomorphic to $K_5$, the complete graph on $5$ vertices, or $K_{3,3}$, the complete bipartite graph on six vertices.  This result is also attributed to Pontryagin (\cite{kennedyhist}).  In their 2001 paper \cite{do2}, Davis and Okun point out that the $K_{3,3}$ graph can be understood as the nerve of a right-angled Coxeter system and prove that this graph is not planar using results from $\ell^2$-homology.  In this paper, we employ a similar method using results from \cite{schroedergeom} to prove $K_5$ is not planar.  
\end{abstract}

\section{Introduction}\label{s:intro}
Let $S$ be a finite set of generators.  A \emph{Coxeter matrix} on $S$ is a symmetric $S\times S$ matrix $M=(m_{st})$ with entries in $\BN\cup\{\infty\}$ such that each diagonal entry is $1$ and each off diagonal entry is $\geq 2$.  The matrix $M$ gives a presentation of an associated \emph{Coxeter group} $W$:
\begin{equation}\label{e:coxetergroup}
	W=\left\langle S\mid (st)^{m_{st}}=1, \text{ for each pair } (s,t) \text{ with } m_{st}\neq\infty\right\rangle.
\end{equation}
The pair $(W,S)$ is called a \emph{Coxeter system}.  Denote by $L$ the nerve of $(W,S)$.  $L$ is a simplicial complex with vertex set $S$, the precise definition will be given in section \ref{s:davis}.  In several papers (e.g., \cite{davisannals}, \cite{davisbook}, and \cite{davismoussong}), M. Davis describes a construction which associates to any Coxeter system $(W,S)$, a simplicial complex $\gS(W,S)$, or simply $\gS$ when the Coxeter system is clear, on which $W$ acts properly and cocompactly.  The two salient features of $\gS$ are that (1) it is contractible and (2) that it admits a cellulation under which the nerve of each vertex is $L$.  It follows that if $L$ is a triangulation of $\BS^{n-1}$, $\gS$ is an aspherical $n$-manifold.  Hence, there is a variation of Singer's Conjecture, originally regarding the (reduced) $\ell^2$-homology of aspherical manifolds, for such Coxeter groups.

\begin{subSCConjecture}\label{conj:singerc} Let $(W,S)$ be a Coxeter group such that its nerve, $L$, is a triangulation of $\BS^{n-1}$.  Then $\cH_{i}(\gS_L)=0$ for all $i\neq\frac{n}{2}$.
\end{subSCConjecture}

For details on $\ltwo$-homology theory, see \cite{davismoussong}, \cite{do2} and \cite{eckmann}.  Conjecture \ref{conj:singerc} holds for elementary reasons in dimensions $1$ and $2$.  In \cite{do2}, Davis and Okun prove that if Conjecture \ref{conj:singerc} holds for \emph{right-angled} Coxeter systems in dimension $n$, then it also holds in dimension $n+1$.  (Here, \emph{right-angled} means generators either commute, or have no relation).  They also prove directly that Conjecture \ref{conj:singerc} holds for right-angled systems in dimension 3, and thus also in dimension 4.  This result also follows from work by Lott and L\"uck (\cite{LL}) and Thurston (\cite{thurston}) regarding Haken Manifolds.  In \cite{schroedergeom}, the author proves that Conjecture \ref{conj:singerc} holds for arbitrary Coxeter systems with nerve $\BS^2$.  

Also in \cite{do2}, Davis and Okun use their low dimensional results to prove the following generalization of Conjecture \ref{conj:singerc}.
\begin{Lemma}\label{l:k33sub}\textup{(Lemma 9.2.3, \cite{do2})} Suppose $(W,S)$ is a right-angled Coxeter system with nerve $L$, a flag triangulation of $\BS^2$.  Let $A$ be a full subcomplex of $L$.  Then 
\[\cH_i(W\gS_A)=0 \text{ for } i>1.\]
\end{Lemma}
Here, $\gS_A$ is the Davis complex associated with the Coxeter system $(W_A,A^0)$, where $W_A$ is the subgroup of $W$ generated by vertices in $A$, with nerve $A$.  It is a subcomplex of $\gS$.  

Lemma \ref{l:k33sub} is the key to what Davis and Okun call ``a complicated proof of the classical fact that $K_{3,3}$ is not planar,''  (See section 11.4.1, \cite{do2}).  We outline that argument in Section \ref{s:previous}.  The purpose of this paper is employ a similar argument to prove that $K_5$ is not planar.  

The key step for us is proving a result analogous to Lemma \ref{l:k33sub}, but for subcomplexes of arbitrary Coxeter systems.  
\begin{Mainthm}\label{t:mainintro} \textup{(See Theorem \ref{t:main})} Let $(W,S)$ be a Coxeter system with nerve $L$, a triangulation of $\BS^2$.  Let $A$ be full subcomplex of $L$ with right-angled complement.  Then 
\[\cH_i(W\gS_A)=0 \text{ for } i>1.\]
\end{Mainthm}
Here $A$ having a ``right-angled complement'' means that for generators $s$ and $t$, the Coxeter relation $m_{st}\neq 2$ nor $\infty$ implies that the vertices corresponding to $s$ and $t$ are both in $A$.  From the Main Theorem, it follows that $K_5$ is not planar.
  
\section{The Davis complex}\label{s:davis}
Let $(W,S)$ be a Coxeter system.  Given a subset $U$ of $S$, define $W_{U}$ to be the subgroup of $W$ generated by the elements of $U$.  A subset $T$ of $S$ is \textit{spherical} if $W_T$ is a finite subgroup of $W$.  In this case, we will also say that the subgroup $W_{T}$ is spherical.  Denote by $\cs$ the poset of spherical subsets of $S$, partially ordered by inclusion.  Given a subset $V$ of $S$, let $\cs_{\geq V}:=\{T\in \cs|V\subseteq T\}$.  Similar definitions exist for $<, >, \leq$.  For any $w\in W$ and $T\in \cs$, we call the coset $wW_{T}$ a \emph{spherical coset}.  The poset of all spherical cosets we will denote by $W\cs$.  

Let $K=|\cs|$, the geometric realization of the poset $\cs$.  It is a finite simplicial complex.  Denote by $\gS(W,S)$, or simply $\gS$ when the system is clear, the geometric realization of the poset $W\cs$.  This is the Davis complex.  The natural action of $W$ on $W\cs$ induces a simplicial action of $W$ on $\gS$ which is proper and cocompact.  $K$ includes naturally into $\gS$ via the map induced by $T\rightarrow W_{T}$.  So we view $K$ as a subcomplex of $\gS$, and note that $K$ is a strict fundamental domain for the action of $W$ on $\gS$.  

The poset $\cs_{>\emptyset}$ is an abstract simplicial complex.  This simply means that if $T\in\cs_{>\emptyset}$ and $T'$ is a nonempty subset of $T$, then $T'\in \cs_{>\emptyset}$.  Denote this simplicial complex by $L$ and call it the \emph{nerve} of $(W,S)$.  The vertex set of $L$ is $S$ and a non-empty subset of vertices $T$ spans a simplex of $L$ if and only if $T$ is spherical.  

Define a labeling on the edges of $L$ by the map $m:\Edge(L)\rightarrow \{2,3,\ldots\}$, where $\{s,t\}\mapsto m_{st}$.  This labeling accomplishes two things: (1) the Coxeter system $(W,S)$ can be recovered (up to isomorphism) from $L$ and (2) the $1$-skeleton of $L$ inherits a natural piecewise spherical structure in which the edge $\{s,t\}$ has length $\pi-\pi/m_{st}$.  $L$ is then a \emph{metric flag} simplicial complex (see Definition \cite[I.7.1]{davisbook}).  This means that any finite set of vertices, which are pairwise connected by edges, spans a simplex of $L$ if an only if it is possible to find some spherical simplex with the given edge lengths.  In other words, $L$ is ``metrically determined by its $1$-skeleton.''  

Recall that a simplicial complex $L$ is \emph{flag} if every nonempty, finite set of vertices that are pairwise connected by edges spans a simplex of $L$.  Thus, it is clear that any flag simplicial complex can correspond to the nerve of a right-angled Coxeter system.  For the purpose of this paper, we will say that labeled (with integers $\geq 2$) simplicial complexes are \emph{metric flag} if they correspond to the labeled nerve of some Coxeter system.  We then treat vertices of metric flag simplicial complexes as generators of a corresponding Coxeter system.  Moreover, for a metric flag simplicial complex $L$, we write $\gS_L$ to denote the associated Davis complex.   

\paragraph{A cellulation of $\gS$ by Coxeter cells.}  $\gS$ has a coarser cell structure: its cellulation by ``Coxeter cells.''  (References include \cite{davisbook} and \cite{do2}.)  The features of the Coxeter cellulation are summarized by \cite[Proposition 7.3.4]{davisbook}.  We point out that under this cellulation the link of each vertex is $L$.  It follows that if $L$ is a triangulation of $\BS^{n-1}$, then $\gS$ is a topological $n$-manifold.  

\paragraph{Full subcomplexes.}  Suppose $A$ is a full subcomplex of $L$.  Then $A$ is the nerve for the subgroup generated by the vertex set of $A$.  We will denote this subgroup by $W_{A}$.  (This notation is natural since the vertex set of $A$ corresponds to a subset of the generating set $S$.)  Let $\cs_{A}$ denote the poset of the spherical subsets of $W_A$ and let $\gS_{A}$ denote the Davis complex associated to $(W_{A},A^{0})$.  The inclusion $W_{A}\hookrightarrow W_{L}$ induces an inclusion of posets $W_{A}\cs_{A}\hookrightarrow W_{L}\cs_{L}$ and thus an inclusion of $\gS_{A}$ as a subcomplex of $\gS_{L}$.  Note that $W_{A}$ acts on $\gS_{A}$ and that if $w\in W_{L}-W_{A}$, then $\gS_{A}$ and $w\gS_{A}$ are disjoint copies of $\gS_{A}$.  Denote by $W_{L}\gS_{A}$ the union of all translates of $\gS_{A}$ in $\gS_{L}$.  

\section{Previous results in $\ell^2$-homology}\label{s:previous}
Let $L$ be a metric flag simplicial complex, and let $A$ be a full subcomplex of $L$.  The following notation will be used throughout.
\begin{align}
\mfh_i(L) &:= \cH_i(\gS_L)\label{e:not1}\\
\mfh_i(A) &:= \cH_i(W_L\gS_A)\label{e:not2}\\
\gb_{i}(A)&:= \dim_{W_L}(\mfh_i(A)).\label{e:not4}
\end{align}
Here $\dim_{W_L}(\mfh_i(A))$ is the von Neumann dimension of the Hilbert $W_L$-module $W_L\gS_A$ and $\gb_{i}(A)$ is the $i^{\text{th}}$ $\ltwo$-Betti number of $W_L\gS_A$.  The notation in \ref{e:not2} and \ref{e:not4} will not lead to confusion since $\dim_{W_L}(W_L\gS_A)=\dim_{W_A}(\gS_A)$.  (See \cite{do2} and \cite{eckmann}).  

\paragraph{$0$-dimensional homology.}\label{p:0-dim}  Let $\gS_A$ be the Davis complex constructed from a Coxeter system with nerve $A$, so $W_A$ acts geometrically on $\gS_A$.  The reduced $\ell^2$-homology groups of $\gS_A$ can be identified with the subspace of \emph{harmonic $i$-cycles} (see \cite{eckmann} or \cite{do2}).  That is, $x\in\mfh_i(A)$ is an $i$-cycle and $i$-cocycle.  $0$-dimensional cocycles of $\gS_A$ must be constant on all vertices of $\gS_A$.  It follows that if $W_A$ is infinite, and therefore the $0$-skeleton of $\gS_A$ is infinite, $\gb_0(A)=0$.  

\paragraph{Singer Conjecture in dimensions 1 and 2.}\label{p:low-dimsinger} As mentioned in Section \ref{s:intro}, Conjecture \ref{conj:singerc} is true in dimensions dimensions 1 and 2.  Indeed, let $L$ be $\BS^0$ or $\BS^1$, the nerve of a Coxeter system $(W,S)$.  Then $W$ is infinite and so, as stated above, $\gb_0(L)=0$.  Poincar\'e duality then implies that the top-dimensional $\ltwo$-Betti numbers are also 0.  

\paragraph{Orbihedral Euler Characteristic.}  $\gS_L$ is a geometric $W$-complex.  So there are only finite number of $W$-orbits of cells in $\gS_L$, and the order of each cell stabilizer is finite. The \emph{orbihedral Euler characteristic} of $\gS_L/W=K$, denoted $\chi^{\text{orb}}(\gS_L/W)$, is the rational
number defined by 
\begin{equation}\label{eqn:chiorb}
\chi^{\text{orb}}(\gS_L/W)=\chi^{\text{orb}}(K)=\sum_{\gs}\frac{(-1)^{\dim\gs}}{|W_\gs|},
\end{equation}
where the summation is over the simplices of $K$ and $|W_\gs|$ denotes the order of the stabilizer in $W$ of $\gs$.
Then, if the dimension of $L$ is $n-1$, a standard argument (see \cite{eckmann}) proves Atiyah's formula:
\begin{equation}\label{e:atiyah}
\chi^{\text{orb}}(K)=\sum_{i=0}^{n}(-1)^i\gb_i(L).
\end{equation}

\paragraph{Joins.}  If $L=L_1\ast L_2$, the join of $L_1$ and $L_2$, where each edge connecting a vertex of $L_1$ with a vertex of $L_2$ is labeled $2$, we write $L=L\ast_2 L_2$ and then $W_L=W_{L_1}\times W_{L_2}$ and $\gS_L=\gS_{L_1}\times\gS_{L_2}$.  We may then use K\"unneth formula to calculate the (reduced) $\ltwo$-homology of $\gS_{L}$, and the following equation from \cite[Lemma 7.2.4]{do2} extends to our situation:
\begin{equation}\label{e:rt-angledjoin}
	\gb_{k}(L_{1}\ast L_{2})=\sum_{i+j=k}\gb_{i}(L_{1})\gb_{j}(L_{2}).
\end{equation}
If $L=P\ast_2 L_2$, where $P$ is one point, then we call $L$ a \emph{right-angled cone}.  $\gS_P=\left[-1,1\right]$, so there are no 1-cycles, and so $\gb_1=(P)=0$.  But, $\chi^{\text{orb}}(\gS_P/W_P)=1/2$ so by equation \ref{e:atiyah}, $\gb_0(P)=1/2$.  Thus, in reference to the right-angled cone $L$, equation \ref{e:rt-angledjoin} implies that 
\begin{equation}\label{e:rt-angledcone}
\gb_i(L)=\frac{1}{2}\gb_i(L_2)
\end{equation}

\paragraph{The $K_{3,3}$ case.}\label{par:k33}
Along with Lemma \ref{l:k33sub}, the above gives us enough to prove that $K_{3,3}$ is not planar.  Indeed, let $P_3$ denote $3$ disjoint points.  Then $K_{3,3}=P_3\ast_2 P_3$ is the nerve of a right-angled Coxeter system.  Since $W_{K_{3,3}}$ is infinite, so $\gb_0(K_{3,3})=0$, and equations \ref{eqn:chiorb} and \ref{e:atiyah} give us that $\gb_1(P_3)=1/2$.  It then follows from equation \ref{e:rt-angledjoin} that $\gb_2(K_{3,3})=1/4$.  Thus, if $K_{3,3}$ were a planar graph, it could be embedded as a full-subcomplex of a flag triangulation of $\BS^2$, where each edge is labeled 2.  This triangulation of $\BS^2$ corresponds to the nerve of a right-angled Coxeter system.  But this contradicts Lemma \ref{l:k33sub}.  For details on this proof see \cite[Sections 8, 9 and 11]{do2}.

\section{The $K_5$ Graph}
Let $K_5$ denote the complete graph on 5 vertices.  The right-angled methods above cannot be applied to $K_5$ because, if the edges are labeled with 2's, then $K_5$ cannot be embedded as a full subcomplex of a metric flag triangulation of $\BS^2$.  However, $K_5$ is metric flag if the edges are labeled with 3's.  For if $r,s$ and $t$ are generators of a Coxeter system such that $m_{rs}=m_{st}=m_{rt}=3$, then $\{r,s,t\}$ is not a spherical subset and this set does not span a $2$-simplex in  the nerve of the correpsonding Coxeter system.  This simple observation leads to the following definition.

\begin{Defn} We say a full subcomplex $A$ of a metric flag simplicial complex $L$ has a \emph{right-angled complement} if the label on all edges not in $A$ is $2$.  
\end{Defn}
The following two Lemmas will be used in the set-up and proof of our Main theorem.

\begin{Lemma}\label{l:full}  Let $L$ be a metric flag simplicial complex, $A\subseteq L$ a full subcomplex with a right-angled complement.  Let $B$ be a full subcomplex of $L$ such that $A\subseteq B$ and let $v\in B-A$ be a vertex.  Then $B_v$, the link of $v$ in $B$, is a full subcomplex of $L$.
\end{Lemma}
\begin{proof} Let $T$ be a subset of vertices contained in $B_v$ and the vertex set of a simplex $\gs$ of $L$.  Then $T$ defines a spherical subset of the corresponding Coxeter system.  Since the  of $T$ are in $B_v$, $v$ commutes with each vertex of $T$. Thus $T\cup\{v\}$ is a spherical subset and therefore $\gs$ is in $B_v$.  
\end{proof}

\begin{Lemma}\label{l:2-dim} Let $L$ be a metric flag triangulation of $\BS^1$, let $A$ be a full subcomplex of $L$.  Then $\gb_i(A)=0$ for $i>1$.
\end{Lemma}
\begin{proof} Consider the long exact sequence of the pair $(\gS_L,W\gS_A)$:
\[0\rightarrow\mfh_2(A)\rightarrow\mfh_2(L)\rightarrow\mfh(L,A)\rightarrow...\]
Since Conjecture \ref{conj:singerc} is true in dimension 2, $\mfh_2(L)=0$ and exactness implies the result.
\end{proof}

For convenience, we restate the relevant result from \cite{schroedergeom} needed to prove $K_5$ is non-planar.
\begin{Theorem}\label{t:basecase}\textup{(See Corollary 4.4, \cite{schroedergeom})} Let $L$ be a metric flag triangulation of $\BS^2$.  Then 
\[\mfh_i(L)=0 \text{ for all } i\]
\end{Theorem}

We are now ready to prove our main theorem, analogous to Lemma \ref{l:k33sub}.    
\begin{Theorem}\label{t:main} Let $L$ be a metric flag triangulation of $\BS^2$, $A\subseteq L$ a full subcomplex with right-angled complement.  Then 
\[\gb_i(A)=0 \text{ for } i>1\]
\end{Theorem}
\begin{proof} Let $B$ be a full subcomplex of $L$ such that $A\subseteq B\subseteq L$.  We induct on the number of vertices of $L-B$, the case $L=B$ given by Theorem \ref{t:basecase}.  Assume $\mfh_i(B)=0$ for $i>1$.   Let $v$ be a vertex of $B-A$ and set $B'=B-v$.  Then $B=B'\cup C_2B_v$ where $B_v$ (by Lemma \ref{l:full}) and $B'$ are full subcomplexes.  We have the following Mayer-Vietoris Sequence:
\[\ldots\rightarrow\mfh_i(B_v)\rightarrow\mfh_i(B')\oplus\mfh_i(C_2B_v)\rightarrow\mfh_i(B)\rightarrow\ldots\]
$B_v$ is a full subcomplex of $L_v$, the link of $v$ in $L$, a metric flag triangulation of $\BS^1$.  So Lemma \ref{l:2-dim} implies $\mfh_i(B_v)=0$, for $i>1$.  Thus, by equation \ref{e:rt-angledcone}, $\mfh_i(C_2B_v)=0$ for $i>1$.  It follows from exactness that $\mfh_i(B')=0$.  
\end{proof}

The above Theorem can be restated as follows, cf. \cite[Theorem 11.4.1]{do2}.
\begin{Theorem}\label{t:planar} Let $A$ be a metric flag complex of dimension $\leq 2$.  Suppose $A$ is planar (that is, it can be embedded as a subcomplex of the $2$-sphere).  Then 
\[\gb_2(A)=0.\]
\end{Theorem}
\begin{proof} By Mayer-Vietoris, we may assume $A$ is connected.  Suppose $A$ is piecewise linearly embedded in $\BS^2$. By introducing a new vertex in the interior of each complementary region, and coning off the boundary of each region labeling each new edge with 2, we obtain a metric flag triangulation of $\BS^2$ in which every edge not in $A$ is labeled 2, that is, $A$ has a right-angled complement.  The result follows from the proof of Theorem \ref{t:main}.
\end{proof}

We are now ready to prove $K_5$ is not planar. 
\begin{Corollary}\label{t:k5notplanar} $K_5$ is not planar.
\end{Corollary}
\begin{proof} Label each edge of $K_5$ with 3, and thus $K_5$ is a metric flag complex.  In this case, $\chi^{\text{orb}}(K_5)=\frac{1}{6}$.  Then Atiyah's formula, equation \ref{e:atiyah}, and the fact that $\gb_0(K_5)=0$ imply that $\gb_2(K_5)>0$, contradicting Theorem \ref{t:planar}.
\end{proof}

\bibliography{mybib}

\begin{thebibliography}{10}

\bibitem{davisannals}
M.~W. Davis.
\newblock Groups generated by reflections and aspherical manifolds not covered
  by {E}uclidean space.
\newblock {\em Annals of Mathematics}, 117:293--294, 1983.

\bibitem{davisbook}
M.~W. Davis.
\newblock {\em The Geometry and Topology of Coxeter Groups}.
\newblock Princeton University Press, Princeton, 2007.

\bibitem{davismoussong}
M.~W. Davis and G.~Moussong.
\newblock Notes on nonpositively curved polyhedra.
\newblock Ohio State Mathematical Research Institute Preprints, 1999.

\bibitem{do2}
M.~W. Davis and B.~Okun.
\newblock Vanishing theorems and conjectures for the $\ell^2$-homology of
  right-angled {C}oxeter groups.
\newblock {\em Geometry \& Topology}, 5:7--74, 2001.

\bibitem{eckmann}
B.~Eckmann.
\newblock Introduction to $\ell^2$-methods in topology: reduced
  $\ell^2$-homology, harmonic chains, $\ell^2$-betti numbers.
\newblock {\em Israel Jounal of Mathematics}, 117:183--219, 2000.

\bibitem{kennedyhist}
J.~W. Kennedy, L.~V. Quintas, and M.~M. Syslo.
\newblock {The Theorem on Planar Graphs}.
\newblock {\em Historia Mathematica}, 12:356--368, 1985.

\bibitem{kuratowski}
K.~Kuratowski.
\newblock {Sur le probl\'eme des courbes gauches en Topologie}.
\newblock {\em Fundamenta Mathematicae}, 15:217--283, 1930.

\bibitem{LL}
J.~Lott and W.~L\"uck.
\newblock $\ltwo$-topological invariants of 3-manifolds.
\newblock {\em Invent. Math.}, 120:15--60, 1995.

\bibitem{schroedergeom}
T.~A. Schroeder.
\newblock {Geometrization of 3-dimensional Coxeter orbifolds and Singer's
  conjecture}.
\newblock {\em {G}eometriae {D}edicata}, 140(1):163ff, 2009.
\newblock DOI number: 10.1007/s10711-008-9314-5.

\bibitem{thurston}
W.~Thurston.
\newblock Three-dimensional manifolds, {K}leinian groups and hyperbolic
  geometry.
\newblock {\em Bull. Amer. Math. Soc.}, 6:357--381, 1982.

\end{thebibliography}

\end{document}